\documentclass[12pt]{amsart}
\usepackage{amssymb}
\usepackage[colorlinks=true]{hyperref}

\newtheorem{theorem}{Theorem}[section]
\newtheorem{corollary}[theorem]{Corollary}
\newtheorem{lemma}[theorem]{Lemma}

\theoremstyle{definition}

\newtheorem{remark}[theorem]{Remark}

\newtheorem*{xrem}{Remark}

\numberwithin{equation}{section}

\newcommand{\bbN}{\mathbb{N}}

\newcommand{\bbQ}{\mathbb{Q}}

\newcommand{\bbZ}{\mathbb{Z}}

\newcommand{\cO}{\mathcal{O}}

\DeclareMathOperator{\disc}{disc}

\DeclareMathOperator{\Gal}{Gal}

\begin{document}

\title{Families of cyclic quartic monogenic polynomials}

\author{Paul M. Voutier}

\dedicatory{In memory of Marie-Nicole Gras}


\begin{abstract}
We produce an explicit family of totally real cyclic quartic polynomials that are
monogenic in many cases and, if the $abc$ conjecture holds, generate distinct
monogenic quartic fields infinitely often. Additional families (also conjecturally
generating infinitely many distinct fields) are provided
in Section~4, including what appears to be an infinite collection of such families.
\end{abstract}



\maketitle

\section{Introduction}

\subsection{Background}


A monic polynomial $f(X) \in \bbZ[X]$ is called a \emph{monogenic polynomial}
if it is irreducible over $\bbQ$ and $\left\{ 1, \theta, \theta^{2}, \ldots, \theta^{\deg(f)-1} \right\}$
is an integral basis for the ring of integers, $\cO_{K}$, of $K=\bbQ(\theta)$,
where $\theta$ is a root of $f(X)$. Similarly, a number field is called a \emph{monogenic
field} if there exists such a power basis for its ring of integers, $\cO_{K}$.

It is well-known that not all number fields are monogenic. The most famous example
being due to Dedekind, the field formed by adjoining a root of $X^{3}-X^{2}-2X-8$
to $\bbQ$. See \cite[pp.~64, 79--81]{Nark} and the references there, as well as
\cite{Gaal} for more information.

Recently, Harrington and Jones \cite{HJ} produced families of quartic polynomials that
are monogenic infinitely often for each of the non-cyclic Galois groups that are
possible for quartic polynomials.

To the best of this author's knowledge, the following are all the known quartic
monogenic fields with $C_{4}$ as their Galois group.

Gras \cite[Proposition~2]{Gr2} showed that there are only two non-real monogenic
cyclic quartic fields $\bbQ \left( \zeta_{5} \right)$ and
$\bbQ \left( \zeta_{16}-\zeta_{16}^{-1} \right)$,
where $\zeta_{n}$ is a primitive $n$-th root of unity. The minimal polynomials of
generators of a power basis for these fields are $X^{4}+X^{3}+X^{2}+X+1$
and $X^{4}+4X^{2}+2$, respectively.
Gras \cite{Gr2} also found $13$ real cyclic quartic fields that are monogenic:
there are $12$ in Table~1 (which includes those arising from $z=2, 3$ in our
family below) and one more in Table~2 (the $j=15$ entry) in \cite{Gr2}.

Finally, there are three additional such fields in LMFDB \cite{LMFDB}:
4.4.80100882173.1 of conductor 12259 (this arises from $z=5$ in our family below);
4.4.839020734032.1 of conductor 35204; and
4.4.1804395776000.1 of conductor 39440 (this arises from $z=6$ in our family below)
-- see \url{https://www.lmfdb.org/NumberField/?count=None&degree=4&signature=%5B4%2C0%5D&galois_group=C4&monogenic=yes}.

\subsection{Results}

Here we produce an analogue in the case of totally real cyclic quartic
polynomials of the results of Harrington and Jones. We use the following
family of polynomials. For any integer $z$, let
\begin{align}
\label{eq:poly1}
f_{z}(X)
= & X^{4} - zX^{3}
- \frac{3z^{6}-16z^{4}+37z^{2}-32}{8} X^{2} \\
& - \frac{2z^{9}-19z^{7}+72z^{5}-135z^{3}+96z}{16} X \nonumber \\
& - \frac{3z^{12}-40z^{10}+214z^{8}-576z^{6}+719z^{4}-64z^{2}-512}{256}. \nonumber
\end{align}

In fact, this is only one of many (infinitely many?) such families. See
Section~\ref{sect:further-families} for more information.

\begin{theorem}
\label{thm:1}
{\rm (a)} For all $z \in \bbZ$, $f_{z}(X)$ is irreducible over
$\bbQ$ and $\Gal \left( f_{z}(X) \right)=C_{4}$.

In addition, if $z^{2}-2$ is squarefree and if $\left( z^{4}-4z^{2}+8 \right)/4$
is squarefree when $z$ is even and if $z^{4}-4z^{2}+8$ is squarefree when $z$ is
odd, then $\bbQ \left( \vartheta \right)$ is a monogenic field, where $\vartheta$
is a root of $f_{z}(X)$.

\vspace*{1.0mm}

\noindent
{\rm (b)} If $z \not\equiv 2 \pmod{4}$, then $f_{z}(X) \in \bbZ[X]$.
If $z$ also satisfies the conditions in {\rm (a)}, then
$\left\{ 1, \vartheta, \vartheta^{2}, \vartheta^{3} \right\}$
is a $\bbZ$-basis for $\cO_{\bbQ(\vartheta)}$.

\vspace*{1.0mm}

\noindent
{\rm (c)}
If $z \equiv 2 \pmod{4}$, then $f_{z}(X-1/2) \in \bbZ[X]$.
If $z$ also satisfies the conditions in {\rm (a)}, then
$\left\{ 1, \vartheta_{1}, \vartheta_{1}^{2}, \vartheta_{1}^{3} \right\}$
is a $\bbZ$-basis for $\cO_{\bbQ(\vartheta_{1})}$, where $\vartheta_{1}
=\vartheta+1/2$ is a root of $f_{z}(X-1/2)$.
\end{theorem}

\begin{xrem}
(1) Note that $f_{-z}(X)$ generates the same field as $f_{z}(X)$, since
$f_{-z}(-X)=f_{z}(X)$.

\vspace*{1.0mm}

(2) If $|z| \geq 2$, then all the roots of $f_{z}(X)$ are real
(see Remark~\ref{rem:roots-real} below).
If $z = \pm 1$, then $f_{z}(X)=(\pm X)^{4}+(\pm X)^{3}+(\pm X)^{2}+(\pm X)+1$
and we obtain $\bbQ \left( \zeta_{5} \right)$.
If $z=0$, then $f_{z}(X)=X^{4}+4X^{2}+2$ and we obtain
$\bbQ \left( \zeta_{16}-\zeta_{16}^{-1} \right)$. So we know from Gras' work
\cite{Gr2} that Theorem~\ref{thm:1} holds for these three values of $z$.

\vspace*{1.0mm}

(3) The square-free conditions on the values of both $z^{2}-2$ and $\left( z^{4}-4z^{2}+8 \right)/4$
are required. For example, in the following cases, $f_{z}(X)$
does not yield a monogenic field:\\
$z=4$ (here $z^{2}-2=14$, but $\left( z^{4}-4z^{2}+8 \right)/4=50$),\\
$z=88$ (here $z^{2}-2=7742$ is divisible by $7^{2}$).

There are also examples where both are not squarefree. E.g., $z=284$,
where $z^{2}-2=80654$ is divisible by $7^{2}$ and
$\left( z^{4}-4z^{2}+8 \right)/4=1626266930$ is divisible by $29^{2}$. The
resulting field is not monogenic.
\end{xrem}

\begin{corollary}
\label{cor:infinite}
Assuming the $abc$ conjecture, there are infinitely many distinct monogenic
cyclic quartic fields.
\end{corollary}

\begin{proof}
This is an immediate consequence of Theorem~\ref{thm:1} above and Theorem~1 in
\cite{Granville}. The latter result shows that there are infinitely many $z$
such that the squarefree conditions hold (under the $abc$ conjecture). The
distinctness follows for $z$ positive satisfying the squarefree condition,
since the fields in Theorem~\ref{thm:1} then contain a quadratic fields of distinct
conductor by Lemma~\ref{lem:smallk-conductor}.
\end{proof}

In fact, the $abc$ conjecture may be more than we require to show there are
infinitely many such fields. What is actually required is that the two polynomials
in $z$ in Theorem~\ref{thm:1}(a) are both squarefree for infinitely many $z$.
For polynomials that are products of irreducible polynomials of degree at most
$3$, this has been known for some time (see \cite{Erdos}), but is an open
problem for polynomials with larger degree irreducible factors (see also \cite{BB}).

Lastly, Maple and PARI/GP \cite{Pari} code used for this work is publicly available
in the \verb!quartic-cyclic! subdirectory at\\
\url{https://github.com/PV-314/monogenic}.
The author is very happy to help interested readers who have any questions, problems
or suggestions for the use of this code.

\section{Construction of our polynomials}
\label{sect:poly}

Our family of polynomials was found by using a result of Gras \cite{Gr2}, so we
begin by recalling some relevant terminology and notation from Section~1 there.

Let $K/\bbQ$ be a cyclic extension of degree $4$, let $G=\langle \sigma \rangle=
\Gal(K/\bbQ)$, and let $k$ be the quadratic subfield of $K$. Let $f$ be the
conductor of $K$. Since $G$ is an abelian group, the Kronecker-Weber theorem states
that $K$ is a subfield of a
cyclotomic field $\bbQ \left( \zeta_{n} \right)$, where $\zeta_{n}$ denotes a
primitive $n$-th root of unity.
If $n$ is the smallest integer for which this holds, the conductor of
$K$ is $n$ if $K$ is fixed by complex conjugation.
If $m$ is the conductor of $k$,
then the conductor of $K$ is of the form $mg$, where $g \in \bbN^{*}$.
If $m$ is odd, then either $g$ is odd, or $g \equiv 4 \pmod{8}$,
or $g \equiv 8 \pmod{16}$.
If $m$ is even, then $g \equiv 2 \pmod{4}$ and
$m \equiv 8 \pmod{16}$.

Let $\chi$ be one of the two Dirichlet characters associated to $K$.
We have $\chi(-1)=\pm 1$ and the field $K$ is real if and only if $\chi(-1)=+1$.

From the bottom of page~1 of \cite{Gr2}, we have $K=\bbQ(\psi)$, where
\[
\psi = \sqrt{\frac{\chi(-1)g \sqrt{m} \left( a+\sqrt{m} \right)}{2}},
\]
for some $a$ and $b$ with $b$ even satisfying $m=a^{2}+b^{2}$.
The algebraic conjugates of $\psi$ are:
\begin{align}
\label{eq:psi-conjs}
\psi^{(1)} & = \psi, \quad
& \psi^{(2)} &= -\psi^{(1)},\\
\psi^{(3)} &= \sqrt{-\chi(-1)g\sqrt{m} \left( a-\sqrt{m} \right) / 2} \hspace*{3.0mm} \text{ and }
& \psi^{(4)} &= -\psi^{(3)}. \nonumber
\end{align}

As on line~2 of page~2 of \cite{Gr2}, we put
\begin{equation}
\label{eq:theta}
\vartheta=\frac{t+z\sqrt{m}+2x\psi+2y\sigma(\psi)}{4},
\end{equation}
where $x$, $y$ and $z$ will be as in Lemma~\ref{lem:1} below. In equation~(0) on
page~2 of \cite{Gr2}, Gras states a necessary and sufficient condition due to
Hasse \cite[Zusatz on page~34]{Ha1} for $\vartheta$ to be an algebraic integer
(Hasse's $f$, $G$, $x_{0}$, $x_{1}$, $y_{0}$, $y_{1}$ correspond to Gras' $m$,
$g$ ,$t$, $z$, $x$, $y$, respectively):
\begin{equation}
\label{eq:t-conditions}
\begin{array}{ll}
\text{(i) if $m$ is odd} & : t \equiv z \pmod{2}, \dfrac{t+z}{2} \equiv gx \pmod{2} \text{ and } \dfrac{t-z}{2} \equiv gy \pmod{2}, \\
\text{(ii) if $m$ is even} & : t \equiv 0 \pmod{4} \text{ and } z \equiv 0 \pmod{2}.
\end{array}
\end{equation}

\begin{xrem}
I have confirmed that this condition is necessary and sufficient when $m$ is even,
under the condition above that
$g \equiv 2 \pmod{4}$ and $m \equiv 8 \pmod{16}$. The necessity and sufficency
of \eqref{eq:t-conditions} also holds when $m$ is odd and $g \equiv 4 \pmod{8}$
or $g \equiv 8 \pmod{16}$. However, when $m$ and $g$ are both odd, one needs to
choose the sign of $a$ and $b$ correctly for the condition
in \eqref{eq:t-conditions} to be necessary and sufficient (see \cite[page~3,
Section~I.i.iii]{Gr1}).
\end{xrem}

In lieu of the conditions in \eqref{eq:t-conditions} above, we will directly use
a quartic polynomial over $\bbQ$ with $\vartheta$ as a root.

\begin{lemma}[Gras~\cite{Gr1}]
\label{lem:gras-poly}
Let $a,b,g,x,y,z,t \in \bbZ$ and let $\vartheta$ be as in \eqref{eq:theta}. Put
$m=a^{2}+b^{2}$ and $f=mg$. Then
$P(\vartheta)=0$ where
\[
P(X)=X^{4}-S_{1}X^{3}+S_{2}X^{2}-S_{3}X+S_{4},
\]
with
\begin{align}
\label{eq:P-coeffs}
   S_{1} &= t,\\
  8S_{2} &= t^{2} + mz^{2} - 2\left( x^{2} + y^{2} \right) f + 2\left( t^{2} - mz^{2} \right), \nonumber \\
 16S_{3} &= \left[ t^{2} + mz^{2} - 2\left( x^{2} + y^{2} \right) f \right] t
            - 2mz \left[ tz - g \left( a \left( x^{2}-y^{2} \right) -2bxy \right) \right], \nonumber \\
256S_{4} &= \left[ t^{2} + mz^{2} - 2\left( x^{2} + y^{2} \right) f \right]^{2}
            - 4m \left[ tz - g \left( a \left( x^{2}-y^{2} \right) -2bxy \right) \right]^{2}. \nonumber
\end{align}
\end{lemma}

\begin{proof}
This is stated in Section~1(v) (``Polynome irr\'{e}ductible
de $\alpha$ sur $\bbQ$'') on page~4 of \cite{Gr1}. Note that $\alpha$ there
(defined in the displayed formula at the bottom of page~3) is $\vartheta$ in
our notation and in the notation in \cite{Gr2}.
\end{proof}

We now quote here Th\'{e}or\`{e}me~1 of \cite{Gr2} translated.

\begin{lemma}[Gras~\cite{Gr2}]
\label{lem:1}
Let $K/\bbQ$ be a cyclic extension of degree $4$ and conductor $f=mg$, where
$m=a^{2}+b^{2}$, with $b \equiv 0 \pmod{2}$, is the conductor of the quadratic
subfield of $K$. The ring $A$ of integers of $K$ has a $\bbZ$-basis $\left\{ 1, \vartheta,
\vartheta^{2}, \vartheta^{3} \right\}$ if and only if there exist $x, y, z \in \bbZ$
such that
\begin{equation}
\label{eq:1}
\left\{
\begin{array}{l}
b \left( x^{2}-y^{2} \right) + 2axy = \pm 2\\
m \left[ z^{2}-\chi(-1)g \left( x^{2}+y^{2} \right) \right]^{2}-4g^{2}
= \pm 16.
\end{array}
\right.
\end{equation}
and $t \in \bbZ$ such that $\vartheta$ is an algebraic integer.
\end{lemma}

\begin{proof}
This is Th\'{e}or\`{e}me~1 of \cite{Gr2} with the addition of our remark regarding
$t$. The end of the proof of Th\'{e}or\`{e}me~1 of \cite{Gr2}, concerning $t$,
appeals to the conditions in equation~(0) there (see equation~\eqref{eq:t-conditions} here), 
which we have discussed in the remark above.
\end{proof}

We are interested in real fields, $K$, here, so we only consider $\chi(-1)=1$.

We searched for quartic cyclic monogenic polynomials using Gras' criterion.

Using PARI/GP \cite{Pari}, for all triples, $(a,b,x)$, satisfying
$1 \leq a \leq 500$, $2 \leq b \leq 20,000$ (with $b$ even) and $1 \leq x \leq 200$,
we found the values of $y$ such that the first equation in \eqref{eq:1} holds.

For such solutions of that first equation, we did search over all $g$
satisfying $1 \leq g \leq 10^{6}$ such that the second equation in \eqref{eq:1}
has an integer solution, $z$.

For integer solutions, $(a,b,g,m,x,y,z)$, of the system of equations in
\eqref{eq:1}, we chose $t$ such that $P(X)$ in Lemma~\ref{lem:gras-poly} is
irreducible over $\bbQ$ and has integer coefficients. Hence $\vartheta$ is an
algebraic integer with $P(X)$ as its minimal polynomial over $\bbZ$.

We then calculated and compared the discriminant of the associated number field
with the discriminant of the minimal polynomial of $\vartheta$.

To perform this search, use the \verb!thm1_check()! function in \verb!pari\gras-checks.gp!
in the github repository.

Among the examples found are the following.

\begin{lemma}
\label{lem:params}
Let $z \in \bbZ$ with $|z| \geq 2$ and put $\chi(-1)=1$. Then
\begin{align}
\label{eq:params}
& a=z^{2}-2, & b=2, & & g=a, & & m=a^{2}+b^{2}=z^{4}-4z^{2}+8, \\
& x=1,       & y=0, & & & \nonumber
\end{align}
satisfy \eqref{eq:1} in Lemma~$\ref{lem:1}$. When $m$ is odd, then $g \equiv 3 \pmod{4}$
and $m \equiv 1 \pmod{4}$.
When $m$ is even, $g \equiv 2 \pmod{4}$ and $m \equiv 8 \pmod{16}$.
\end{lemma}

\begin{proof}
Equation~\eqref{eq:1} in Lemma~$\ref{lem:1}$ requires just the following simple
verifications: $b \left( x^{2}-y^{2} \right) + 2axy=2$ and
\begin{align*}
& m \left[ z^{2}-\chi(-1)g \left( x^{2}+y^{2} \right) \right]^{2}-4g^{2} \\
= & \left( z^{4}-4z^{2}+8 \right) \left[ z^{2}-\left( z^{2}-2 \right) \right]^{2} - 4\left( z^{2}-2 \right)^{2} \\
= & \left( 4z^{4}-16z^{2}+32 \right) - \left( 4z^{4}-16z^{2}+16 \right)=16.
\end{align*}

The statements involving the parity of $m$ follow immediately, as $a$, and hence
$z$, has the same parity as $m$.
\end{proof}

We will also choose
\begin{equation}
\label{eq:t-values}
t =
\left\{
\begin{array}{cl}
z+2 & \text{if $z \equiv 2 \pmod{4}$}, \\
  z & \text{otherwise.}
\end{array}
\right.
\end{equation}

\begin{lemma}
\label{lem:poly}
Let $a$, $b$, $g$, $m$, $x$, $y$ and $z$ be as in Lemma~$\ref{lem:params}$ and
put $\chi(-1)=1$. With $t=z$, $\vartheta$ in \eqref{eq:theta} is a root of
the polynomial, $f_{z}(X)$, defined in \eqref{eq:poly1}.
\end{lemma}

\begin{proof}
This follows by using Lemma~\ref{lem:gras-poly}.

See the \verb!check_fz_polys()! function in \verb!maple\poly-families.txt!
and in \verb!pari\poly-families.gp! in the github repository for the Maple
and PARI/GP code used here.
\end{proof}

\begin{remark}
\label{rem:roots-real}
For $|z| \geq 2$, since $\sqrt{m}>a>0$ and $\chi(-1)=1$, it follows that $\psi$
and its conjugates are all real, using the expressions for them in
\eqref{eq:psi-conjs} above. Hence $\vartheta$
and all of its conjugates are also real.
\end{remark}

Theorem~\ref{thm:1} follows immediately from Lemmas~\ref{lem:integer-coeffs},
\ref{lem:irreduc}, \ref{lem:galois}, \ref{lem:bigK-conductor} and \ref{lem:smallk-conductor}
in the next section.

\section{Further Lemmas}

\begin{lemma}
\label{lem:integer-coeffs}
Let $z \in \bbZ$.

\noindent
{\rm (a)} $f_{z}(X) \in \bbZ[X]$ if and only if $z \not\equiv 2 \pmod{4}$.

\noindent
{\rm (b)} $f_{z}(X-1/2) \in \bbZ[X]$ if and only if $z \equiv 2 \pmod{4}$.
\end{lemma}

\begin{proof}
(a) For each of the $X^{2}$, $X$ and constant coefficients, we examine the numerator
of the coefficient modulo its denominator using Maple or PARI/GP.
For each coefficient, we find that this congruence class is $0$ modulo the denominator
when $z$ is odd. The congruence classes are $4z_{1}^{2}$, $8z_{1}^{3}$ and $16z_{1}^{4}$
for $X^{2}$, $X$ and the constant coefficients, respectively, when $z=2z_{1}$.
These are $0$ if and only if $z_{1}$ is even, i.e., $4|z$.

\vspace*{1.0mm}

(b) We proceed in the same way here, using Maple or PARI/GP, except here we check
$z$ odd, $4|z$ and $z \equiv 2 \pmod{4}$ separately.

See the \verb!fz_integer_coeff_check()! (for part~(a)) and \verb!fz12_integer_coeff_check()!
(for part~(b)) functions in \verb!maple\poly-checks.txt! and \verb!pari\poly-checks.gp!
in the github repository for the Maple and PARI/GP code used here.
\end{proof}

\subsection{Irreducibility}

We will apply the following result due to Driver, Leonard and Williams \cite{DLW}.

\begin{lemma}
\label{lem:irred-DLW}
Let $r$ and $s$ be integers such that $r^{2}-4s$ is not a perfect square. Then
$f(X)= X^{4}+rX^{2}+s$ is reducible in $\bbZ[X]$ if and only if there exists an
integer $c$ such that $c^{2}=s$ and $2c-t$ is a perfect square.
\end{lemma}

\begin{proof}
This is Corollary~2 of \cite{DLW}.
We have replaced $x$ there with $X$ here, since we
already use $x$ here.
\end{proof}

\begin{lemma}
\label{lem:irreduc}
For all $z \in \bbZ$, $f_{z}(X)$ is irreducible in $\bbQ[X]$.
\end{lemma}

\begin{proof}
For $z=-1,0,1$, we verify this manually. The \verb!z_check()! function in
\verb!pari\small-z-checks.gp! in our github repository can be used for this.
So we may assume here that $|z| \geq 2$ and hence that $\chi(-1)=1$.

From equation~\eqref{eq:theta} with $g=a$, $x=1$, $y=0$ and $t=z$ (from
Lemmas~\ref{lem:params} and \ref{lem:poly}), we have
\[
\vartheta=\frac{z+z\sqrt{m}+2\psi}{4},
\text{ where }
\psi=\sqrt{\frac{a\sqrt{m}(a+\sqrt{m})}{2}}.
\]

Since $2\vartheta-\psi$ is an algebraic number of degree at most $2$ over $\bbQ$,
the irreducibility of $f_{z}(X)$ will follow if we can show that
\[
X^{4}-amX^{2}-\frac{a^{2}m\left( a^{2}-m \right)}{4},
\]
which has $\psi$ as a root, is irreducible in $\bbQ[X]$.

With the choice of $a$ and $m$ in Lemma~\ref{lem:params}, this becomes
\begin{equation}
\label{eq:psi-poly}
f_{\psi}(X)=X^{4}-\left( z^{2}-2 \right) \left( z^{4}-4z^{2}+8 \right) X^{2}+\left( z^{2}-2 \right)^{2} \left( z^{4}-4z^{2}+8 \right).
\end{equation}

In the notation of Lemma~\ref{lem:irred-DLW}, we have
\[
r=-\left( z^{2}-2 \right) \left( z^{4}-4z^{2}+8 \right),
\text{ and }
s=\left( z^{2}-2 \right)^{2} \left( z^{4}-4z^{2}+8 \right),
\]
so that
\[
r^{2}-4s=\left( z^{2}-2 \right)^{4} \left( z^{4}-4z^{2}+8 \right).
\]

This is a square only when $z^{4}-4z^{2}+8=\left( z^{2}-2 \right)^{2}+4$ is a
square. The only squares separated by $4$ are $0$ and $4$. This requires that
$z^{2}=2$, which has no integer solution. Hence $r^{2}-4s$ is never a square
and Lemma~\ref{lem:irred-DLW} applies. It requires that $s$ is a perfect square
if $X^{4}+rX^{2}+s$ is reducible. But we just saw that $z^{4}-4z^{2}+8$ is never
a square, so $s$ is never a square.
Thus $f_{\psi}(X)$, and so $f_{z}(X)$ too, is irreducible in $\bbQ[X]$.
\end{proof}

\subsection{Galois groups}

We will apply the following result due to Kappe and Warren \cite{KW}.

\begin{lemma}
\label{lem:KW}
Let $f(X)=X^{4}+aX^{3}+bX^{2}+cX+d \in \bbQ[X]$ be irreducible over $\bbQ$.
Let
\[
r(X)=X^{3}-bX^{2}+(ac-4d)X - \left( a^{2}d-4bd+c^{2} \right)
\]
with splitting field $E$. Then $\Gal(f) \simeq C_{4}$ if and only if $r(X)$ has
exactly one root $s\in \bbQ$ and
\begin{equation}
\label{eq:gr}
g(X) = \left( X^{2}-sX+d \right) \left( X^2+aX+(b-s) \right)
\end{equation}
splits in $E$.
\end{lemma}

\begin{proof}
This is Theorem~1(iv) of \cite{KW} applied to the special case of $K=\bbQ$, but
we have replaced $t$ there with $s$ here and $x$ there with $X$ here, since we
already use both $t$ and $x$ here.
\end{proof}

\begin{lemma}
\label{lem:galois}
For all $z \in \bbZ$, $\Gal \left( f_{z}(X) \right)=C_{4}$.
\end{lemma}

\begin{proof}
As in the proof of Lemma~\ref{lem:irreduc}, we check this manually for $z=-1,0,1$.
(e.g., using the \verb!z_check()! function in \verb!pari\small-z-checks.gp!)
and assume in the remainder of the proof that $|z| \geq 2$.

From the expression for $\vartheta$ at the start of the proof of Lemma~\ref{lem:poly},
it suffices to show that
$\Gal \left( f_{\psi} \right)=C_{4}$, where
$f_{\psi}(X)$ is defined in \eqref{eq:psi-poly}.

We saw in the proof of Lemma~\ref{lem:irreduc} that $f_{\psi}(X)$ is irreducible
over $\bbQ$. So we can apply Lemma~\ref{lem:KW} to $f_{\psi}(X)$.

Using Maple, it can be shown
that for our polynomials $f_{\psi}(X)$, we can write $r(X)$ as
$X- \left( 16-16z^{2}+6z^{4}-z^{6} \right)$ times
\[
X^{2}
- \left( 4z^{8}-32z^{6}+112z^{4}-192z^{2}+128 \right).
\]

The discriminant of the quadratic factor is
\[
16 \left( z^{2}-2 \right)^{2} \left( z^{4}-4z^{2}+8 \right).
\]

As we saw in the proof of Lemma~\ref{lem:irreduc}, $z^{4}-4z^{2}+8$ is never a
square for $z \in \bbZ$. So the quadratic factor splits in
$E=\bbQ \left( \sqrt{z^{4}-4z^{2}+8} \right)$ and $r(X)$ has only one rational
root, the root of the linear factor, $s=16-16z^{2}+6z^{4}-z^{6}$.

Again, using Maple, we find that
\[
\disc \left( X^{2}-sX+d \right)
= \left( z^{2}-2 \right)^{4} \left( z^{4}-4z^{2}+8 \right)
\]
and
\[
X^{2}+aX+(b-s)=X^{2}.
\]

So $g(X)$ splits in $E$.
Hence, $f_{\psi}(X)$ has $C_{4}$ as its Galois group over $\bbQ$.

See the \verb!galois_psi_check()! function in \verb!maple\poly-checks.txt! or
\verb!pari\poly-checks.gp! in the github repository for the Maple and PARI/GP
code used for the statements above.
\end{proof}

\subsection{Conductor of $K$}

We start with a result due to Spearman and Williams \cite{SW2} that we will
use to determine the conductor.

\begin{lemma}
\label{lem:generator}
Let $K$ be a cyclic extension of $\bbQ$ of degree $4$.
There exist unique integers $A$, $B$, $C$ and $D$ such that
$K = \bbQ \left( \sqrt{A \left( D + B \sqrt{D} \right)} \right)$, where $A$ is
squarefree and odd, $D=B^{2}+C^{2}$ is squarefree, $B>0$, $C>0$ and $\gcd(A,D)=1$.
\end{lemma}

\begin{proof}
This is \cite[Theorem~1]{HHRWH}.
\end{proof}

\begin{lemma}
\label{lem:SW}
Let $K$ be a cyclic extension of $\bbQ$ of degree $4$ and let $A$, $B$, $C$ and $D$
be as in Lemma~$\ref{lem:generator}$. Put
\[
\ell=
\left\{
\begin{array}{ll}
3 & \text{if $D \equiv 2 \pmod{4}$ or if $(D \equiv 1 \pmod{4}$ and $B \equiv 1 \pmod{2})$,}\\
2 & \text{if $D \equiv 1 \pmod{4}$, $B \equiv 0 \pmod{2}$ and $A+B \equiv 3 \pmod{4}$,}\\
0 & \text{if $D \equiv 1 \pmod{4}$, $B \equiv 0 \pmod{2}$ and $A+B \equiv 1 \pmod{4}$.}
\end{array}
\right.
\]

The conductor of $K$ is $2^{\ell}|A|D$.
\end{lemma}

\begin{proof}
This is the main Theorem in \cite{SW2}. In fact, as the authors write in \cite{SW2},
their Theorem was proven earlier by
the second author and others in Theorem~5 of \cite{HHRWH}.
\end{proof}

\begin{lemma}
\label{lem:bigK-conductor}
Let $z \in \bbZ$ with $|z| \geq 2$. The conductor of $K$ is $mg$ if $z^{2}-2$ is
squarefree and $\left( z^{4}-4z^{2}+8 \right)/4$ is squarefree when $z$ is even
and $z^{4}-4z^{2}+8$ is squarefree when $z$ is odd.
\end{lemma}

\begin{proof}
Recall from Section~\ref{sect:poly} that we can write $K=\bbQ(\psi)$, where
\begin{equation}
\label{eq:psi}
\psi= \sqrt{\frac{g \left( m+a\sqrt{m} \right)}{2}},
\end{equation}
since we have $\chi(-1)=1$ when $|z| \geq 2$.

We use this expression for $\psi$ to determine $A$, $B$, $C$ and $D$ in order to
apply Lemma~\ref{lem:SW}.

\vspace*{1.0mm}

(i) First we suppose that $z$ is even.

From Lemma~\ref{lem:params}, we have $m \equiv 8 \pmod{16}$, so we
can write $m=8m_{1}$ with $m_{1}$ odd. Since $m=a^{2}+4$, we must have $v_{2}(a)=1$.
Hence $a=2a_{1}$ with $a_{1}$ odd. Since $g=a$ (again, see Lemma~\ref{lem:params})
we can also write $g=2g_{1}$ with $g_{1}$ odd.

Therefore,
\begin{align*}
\frac{g \left( m+a\sqrt{m} \right)}{2}
&=\frac{2g_{1} \left( 8m_{1}+2a_{1}\sqrt{8m_{1}} \right)}{2}
=\frac{2g_{1} \left( 4 \cdot 2m_{1}+4a_{1}\sqrt{2m_{1}} \right)}{2} \\
&=4g_{1} \left( 2m_{1}+a_{1}\sqrt{2m_{1}} \right).
\end{align*}

Hence $K=\bbQ \left( \psi_{1} \right)$, where
\[
\psi_{1}=\sqrt{g_{1} \left( 2m_{1}+a_{1}\sqrt{2m_{1}} \right)}.
\]

Now we put $A=g_{1}=a_{1}$, $B=a_{1}$, $C=1$ and $D=2m_{1}$. We saw above that
$a_{1}$ is odd, so if $a=2a_{1}=z^{2}-2$ is squarefree, then $A$ is squarefree
and odd, as required. Since $|z| \geq 2$, it follows that $B$ and
$C$ are positive, as required. Also $D=2m_{1}=m/4=(a/2)^{2}+1^{2}=B^{2}+C^{2}$.

We can write $m=\left( z^{2}-2 \right)^{2}+4=a^{2}+4$, so $\gcd (a,m)|4$. Since
both $a$ and $m$ are even, but $a=2a_{1}$ with $a_{1}$ odd, we have $\gcd(a,m)=2$.
Hence $\gcd(A,D)=\gcd \left( a_{1}, 2m_{1} \right)=\gcd \left( a_{1}, m_{1} \right)=1$.
So if $D=m/4$ is squarefree, then $A$, $B$, $C$ and $D$ are as required.

Since $D=2m_{1} \equiv 2 \pmod{4}$ (since $m_{1}$ is odd) and provided both
$a=z^{2}-2$ and $m/4=\left( z^{4}-4z^{2}+8 \right)/4$ are squarefree, we have
$\ell=3$ in Lemma~\ref{lem:SW} and hence the conductor of $K$ is $8(g/2)(m/4)=mg$.

This completes the proof when $z$ is even.

\vspace*{1.0mm}

(ii) Now suppose that $z$ is odd.

Then $m=z^{4}-4z^{2}+8$ is also odd, as are $a=z^{2}-2$ and $g=a$.

Therefore,
\[
\frac{g \left( m+a\sqrt{m} \right)}{2}
=\frac{2a \left( m+a\sqrt{m} \right)}{4}.
\]

Hence $K=\bbQ \left( \sqrt{2a \left( m+2a\sqrt{m} \right)} \right)$.
However, we want a generator for $K$ in the form provided by Lemma~\ref{lem:generator}.
The generator here has $A=2a$ even, whereas Lemma~\ref{lem:generator} states that
a generator exists with $A$ odd. We will show that
\[
\psi_{2}=\sqrt{a \left( m+2\sqrt{m} \right)}
\]
is such a generator. We do so by observing that
\begin{align*}
\left( \frac{\psi_{2}}{\sqrt{2a \left( m+2a\sqrt{m} \right)}} \right)^{2}
& =\frac{a \left( m+2\sqrt{m} \right)}{2a \left( m+a\sqrt{m} \right)}
=\frac{m-2a+(2-a)\sqrt{m}}{2(m-a^{2})}\\
& =\frac{a^{2}-2a+4+(2-a)\sqrt{a^{2}+4}}{8},
\end{align*}
using $m=a^{2}+4$ and where the second equality holds from multiplying the
numerator and denominator of the second expression above by $m-a\sqrt{m}$. The
numerator of the last expression above is
\[
\frac{\left( a-2-\sqrt{a^{2}+4} \right)^{2}}{2}.
\]

Hence $\psi_{2}$ is $\sqrt{2a \left( m+2a\sqrt{m} \right)}$ times an element of
$\bbQ \left( \sqrt{m} \right)$. Therefore, $\psi_{2}$ and $\sqrt{2a \left( m+2a\sqrt{m} \right)}$
generate the same field, $K$.

Therefore, we can put $A=a$, $B=2$, $C=a$ and $D=m=a^{2}+4=B^{2}+C^{2}$.
We saw above that $a$ is odd, so if $a=z^{2}-2$ is squarefree, then
$A$ is squarefree and odd, as required in Lemma~\ref{lem:generator}.
Since $|z|>1$, $B$ and $C$ are positive, as required, and $D=B^{2}+C^{2}$.

We can write $m=\left( z^{2}-2 \right)^{2}+4=a^{2}+4$, so $\gcd (a,m)|4$. Since
both $a$ and $m$ are odd, we have $\gcd(a,m)=1$.
Hence $\gcd(A,D)=\gcd \left( a, m \right)=1$.
So if $D=m=z^{4}-4z^{2}+8$ is squarefree, then $A$, $B$, $C$ and $D$ are as required.

Since $D=m \equiv 1 \pmod{4}$, $B$ is even and $A+B=a+2=z^{2} \equiv 1 \pmod{4}$,
we have $\ell=0$
in Lemma~\ref{lem:SW}. Provided that both $a=z^{2}-2$ and $m=z^{4}-4z^{2}+8$ are
squarefree, we find that the conductor of $K$ is $am=mg$.

This completes the proof when $z$ is odd.
\end{proof}

\subsection{Conductor of $k$}

Recall from Section~\ref{sect:poly} that $k$ is the unique quadratic subfield
of $K$. One of the conditions in Lemma~\ref{lem:1} is that the conductor of $k$
is $m$. We prove that here.

\begin{lemma}
\label{lem:smallk-conductor}
The conductor of $k$ is $m=z^{4}-4z^{2}+8$ if
$m/4$ is squarefree when $z$ is even and
if $m$ is squarefree when $z$ is odd.
\end{lemma}

\begin{proof}
From the expression for $\psi$ in \eqref{eq:psi}, $K$ contains $\bbQ \left( \sqrt{m} \right)$.

For $z$ even, we have $m \equiv 8 \pmod{16}$ from Lemma~\ref{lem:params}. Hence
$m/4 \equiv 2 \pmod{4}$.
Since $m/4$ is squarefree, it follows that $\disc \left( \bbQ \left( \sqrt{m} \right) \right)=4(m/4)=m$.
For real quadratic fields, the conductor equals the discriminant (see Example~3.11
on page~159 of \cite{Mi}), so here the conductor is $m$, as
required.

For $z$ odd, we have $m \equiv 1 \pmod{4}$ from Lemma~\ref{lem:params}. Here $m$
is squarefree, so
$\disc \left( \bbQ \left( \sqrt{m} \right) \right)=m$, and the conductor is $m$
in this case too.
\end{proof}

\section{Further Families}
\label{sect:further-families}

(1) There is also a family of monogenic polynomials that arise
from putting $a=z^{2}+2$, rather than $a=z^{2}-2$, for the choice of parameters
in Lemma~\ref{lem:params}.

For computing the polynomials and checking them, see the function \verb!check_further_family1()!
in either \verb!maple\poly-families.txt! or \verb!pari\poly-families.gp!
in the github repository.

\vspace*{3.0mm}

(2) Another family of monogenic polynomials arises from
$a=1$, $x=y=1$, $z=4v+2$, $t=z$, $g=8v^{2}+8v+4$, $b=g/2$ and $m=a^{2}+b^{2}=g^{2}/4+1
=16v^{4}+32v^{3}+32v^{2}+16v+5$, when $g/4$ and $m$ are both squarefree.
It results in the more complicated polynomial
\begin{align*}
& X^{4} - (4v+2)X^{3} - \left( 96v^{6}+288v^{5}+424v^{4}+368v^{3}+200v^{2}+64v+11 \right)X^{2} \\
& + \left( 512v^{9}+2304v^{8}+5312v^{7}+7840v^{6}+8080v^{5}+5976v^{4}+3176v^{3}
+1172v^{2}+276v+32 \right) X \\
& -768v^{12}-4608v^{11}-14208v^{10}-28800v^{9}-42224v^{8}-46784v^{7} \\
& -39984v^{6}-26480v^{5}-13468v^{4}-5128v^{3}-1386v^{2}-238v-19.
\end{align*}

For computing these polynomials, see the function \verb!check_further_family2()!
in either \verb!maple\poly-families.txt! or \verb!pari\poly-families.gap!
in the github repository.

\vspace*{3.0mm}

(3) We found four families of cyclic quartic monogenic polynomials arising from
$x=3$ and $y=4$.\\
$a=7/15625z^{2}-62/625$, $b=24/15625z^{2}-34/625$, $g=z^{2}/25-2$, where $z=3125v \pm 1020$.\\
$a=7/15625z^{2}+62/625$, $b=24/15625z^{2}+34/625$, $g=z^{2}/25+2$, where $z=3125v \pm 1265$.

They will be monogenic when $g$ and either $m$ or $m/4$ are squarefree. Such
squarefree values do occur. The polynomials for $m=a^{2}+b^{2}$ are of degree $4$
in $v$ and irreducible (see Corollary~\ref{cor:infinite} and the comments after it
for the significance of this).

We also have PARI/GP and Maple code for computing them. See the functions
\verb!check_x3y4_family1()!, \verb!check_x3y4_family2()!,
\verb!check_x3y4_family3()! and \verb!check_x3y4_family4()!
in either \verb!maple\poly-families.txt! or \verb!pari\poly-families.gp!
in the github repository.

\begin{xrem}
Gras also proved in \cite[Proposition~1]{Gr2} that a necessary condition for a
cyclic quartic field to be monogenic is that $g^{2} \pm 4=mc^{2}$ for some integer,
$c$. In each of the above examples, $c$ is fixed, with $c=x^{2}+y^{2}$. This
prompted us to search for other pairs, $(x,y)$, which gave rise to families of
monogenic polynomials with $c=x^{2}+y^{2}$.
For this, we used the \verb!general_xy_search()! function in \verb!pari\xy-search.gp!
in the github repository.

For pairs $\left( x, y \right)=\left( x, \left( x^{2}-1 \right)/2 \right)$
with $x$ odd, we found families of cyclic quartic monogenic fields (each
infinite under the $abc$ conjecture) provided that
$10x^{7}+42x^{5}+70x^{3}+70x-32$ is a quadratic residue $\bmod{\,\left( x^{2}+1 \right)^{4}}$.
Such values of $x$ include $11,15$, $19$, $23$, $25$, $33,35,39,41,43, \ldots$.
The Maple code in the function \verb!x_check()! in
\verb!maple\z2-search.txt! in the github repository can be used to
obtain explicit expressions for $a$, $b$, $g$, $m$ and $z$ for any such value
of $x$.

For pairs $\left( x, y \right)=\left( x, (x/2)^{2}-1 \right)$ with $4|x$, there
is a similar phenomenon. If $x \equiv 2 \pmod{4}$, then $y$ is also even and so
the first equation in \eqref{eq:1} cannot hold, as $4$ divides the left-hand side.

So it appears that not only may there be infinitely many quartic cyclic monogenic
fields, but infinitely many infinite families of them.
\end{xrem}

\subsection*{Acknowledgements}

I thank the referee for their helpful suggestions, as well as some corrections,
which have improved the paper. I am grateful to both the referee and Professor
Tano\'{e} for their encouragement to also treat $z \equiv 2 \pmod{4}$ in
Theorem~\ref{thm:1}. Lastly, I wish to express my condolences to G. Gras, as
well as my appreciation to him for the stimulating email exchanges regarding his
deceased wife's work.

\end{document}